\documentclass[11pt,twoside]{article}
\usepackage{amsfonts}
\usepackage{amsmath}
\usepackage{amssymb}
\usepackage{multicol}
\usepackage{graphicx}
\usepackage{stmaryrd}
\usepackage{cite}
\usepackage{epsfig}

\newtheorem{theorem}{Theorem}[section]
\newtheorem{lemma}[theorem]{Lemma}

\newtheorem{definition}[theorem]{Definition}

\numberwithin{equation}{section}
\newenvironment{proof}[1][Proof]{\noindent\textbf{#1.} }{\hfill $\Box$}
\allowdisplaybreaks

 \makeatletter
\begin{document}
\author{Shuxia Pan\thanks{E-mail: shxpan@yeah.net.  Supported by NSF of Gansu
Province of China (1208RJYA004) and the Development Program for Outstanding Young Teachers in Lanzhou University of Technology (1010ZCX019).} and Jie Liu \\
{Department of Applied Mathematics, Lanzhou University of Technology,}\\
{Lanzhou, Gansu 730050, People's Republic of China}}
\title{\textbf{Existence of Bistable Waves in a Competitive Recursion System with Ricker Nonlinearity}} \maketitle
\begin{abstract}
Using an abstract scheme of monotone semiflows, the existence of bistable traveling wave solutions of a competitive recursion system with Ricker nonlinearity is established. The traveling wave solutions formulate the strong inter-specific actions between two competitive species.

\textbf{Keywords}: monotone semiflows; strong competition; spreading speed.

\textbf{AMS Subject Classification (2010)}: 37C65; 39A12; 92D25.
\end{abstract}

\section{Introduction}
\noindent

It is well known that many objects have multi stable steady states, e.g., water has stable steady states including the liquid  and the solid. Is there a transition process between different steady states? To model the process, we could use the index of traveling wave solutions. In particular, bistable traveling wave solutions can reflect the spatio-temporal processes of the transition, we refer to Volpert et al. \cite{volpert} for some examples in reaction-diffusion models. Besides the bistable traveling wave solutions in parabolic systems,  the bistable traveling wave solutions in other evolutionary systems are also important in describing some facts, for example, in lattice differential equations \cite{keener}. In literature, the existence of bistable traveling wave solutions of scalar equations has been widely studied, we refer to \cite{aron1,bates2,chen,fife,volpert} and the references cited therein.

For coupled systems, the existence of bistable traveling wave solutions is more difficult than that of scalar equations. For example, we cannot use the comparison principle of monotone semiflows in some coupled systems, but the comparison principle holds in many scalar equations. In some cooperative systems, the existence of traveling wave solutions has been proved if the corresponding traveling wave system is finite dimensional, see \cite{huang,mis,volpert} and the reference cited therein. When the traveling wave system is infinite dimensional, then the corresponding study is significantly different, e.g., it is difficult to utilize the technique of phase space analysis. Very recently, Fang and Zhao \cite{fang} have established an abstract scheme to prove the existence of bistable traveling wave solutions of evolutionary systems generating monotone semiflows by developing Yagisita \cite{y}. The theory can be applied to a system even if the corresponding traveling wave system is infinite dimensional. In particular, by the theory in \cite{fang}, Zhang and Zhao \cite{zhang1,zhang2} obtained the existence of bistable traveling wave solutions in some coupled systems.

In this paper, we shall consider the bistable traveling wave solutions of the following recursion system with Ricker nonlinearity
\begin{equation}\label{1}
\begin{cases}
U_{n+1}(x)=\int_{\mathbb{R}}
U_{n}(y)e^{r_{1}(1-U_{n}(y)-a_{1}V_{n}(y))}l_{1}(x-y)dy,\\
V_{n+1}(x)=\int_{\mathbb{R}}
V_{n}(y)e^{r_{2}(1-V_{n}(y)-a_{2}U_{n}(y))}l_{2}(x-y)dy,
\end{cases}
\end{equation}
where $r_1>0,r_2>0, x\in\mathbb{R},n\in\mathbb{N},$  $U_n(x), V_n(x)$ denote the density
of two competitor at time $n$
and in location $x\in \mathbb{R}$ in population dynamics, and  $l_i, i=1,2$ are probability functions describing the dispersal of individuals. For \eqref{1}, Wang and Castillo-Chavez \cite{wc} considered its monostable  traveling wave solutions and spreading speeds for $a_1<1<a_2$. Li and Li \cite{li} further studied the properties of monostable traveling wave solutions when  $a_1<1<a_2$ in \eqref{1}. If $a_1,a_2>1,$ then the corresponding difference system, namely,
\begin{equation}\label{2}
\begin{cases}
u_{n+1}=
u_{n}e^{r_{1}(1-u_{n}-a_{1}v_{n})},n\in\mathbb{N},\\
v_{n+1}=
v_{n}e^{r_{2}(1-v_{n}-a_{2}u_{n})},n\in\mathbb{N}
\end{cases}
\end{equation}
has four equilibria
\[
E_0=(0,0), E_1=(1,0), E_2=(0,1), E_3=(k_1,k_2)
\]
defined by
\[
k_1=\frac{1-a_1}{1-a_1a_2}, k_2=\frac{1-a_2}{1-a_1a_2}.
\]
In particular, if $r_1,r_2\in (0,1],$ then both $E_1$ and $E_2$ are stable while $E_0,E_3$ are unstable. For the global dynamics of \eqref{2},  see Cushing et al. \cite{cush}, Kang and Smith \cite{kang}, Li et al. \cite{lijia}, Sacker \cite{saker}, Sun et al. \cite{sun}.

When $E_1,E_2$ are stable in \eqref{2}, then a traveling wave solution connecting them is a bistable traveling wave solution of \eqref{1}, and a traveling wave solution connecting $E_0(\text{or }E_3)$ with $E_1(\text{or }E_2)$  is a monostable traveling wave solution of \eqref{1}, see \cite{li,wc}. In this paper, we shall prove the existence of bistable traveling wave solutions of \eqref{1} by the theory in Fang and Zhao \cite{fang}, during which the spreading speeds of several monostable subsystems of \eqref{1} are established by the results in Hsu and Zhao \cite{hsuzhao}, Liang and Zhao \cite{liang} and Weinberger et al. \cite{weinberger}.

\section{Preliminaries}
\noindent

In this paper, we shall use the standard partial ordering and ordering interval in $\mathbb{R}$ or $\mathbb{R}^2.$ Let $\mathcal{C}:=C(\mathbb{R},\mathbb{R}^2)$ be
\[
C(\mathbb{R},\mathbb{R}^2)=\{U | U:\mathbb{R}\to \mathbb{R}^2 \text{ is a uniformly continuous and bounded function}\}
\]
equipped with the standard compact open topology, namely,  a sequences $U_n\in \mathcal{C}$ converges to $U\in \mathcal{C}$ if and only if $U_n(x)\in \mathcal{C}$ converges to $U(x)\in \mathcal{C}$ uniformly in any compact subset of $x\in\mathbb{R}$. If $U=(u_1(x),u_2(x)),V=(v_1(x),v_2(x))\in \mathcal{C},$ then
\[
U\ge (\le) V \Leftrightarrow u_i(x) \ge (\le) v_i(x), i=1,2, x\in\mathbb{R},
\]
\[
U> (<) V \Leftrightarrow U\ge (\le) V \text{ but }u_i(x)\neq v_i(x) \text{ for some } i=1,2 \text{ and }x\in\mathbb{R},
\]
and
\[
U\gg (\ll) V \Leftrightarrow U\ge (\le) V \text{ and }u_i(x) > (<) v_i(x), i=1,2, x\in\mathbb{R}.
\]
Moreover, if ${A},{B}\in \mathbb{R}^2$ with ${A}\le {B},$ then
\[
\mathcal{C}_{[{A},{B}]}=\{U: U\in \mathcal{C}, {A}\le U (x)\le {B},x\in\mathbb{R}\}.
\]

To study the bistable traveling wave solutions of \eqref{1}, we shall impose the following assumptions in this paper:
\begin{description}
\item[(H1)] $r_1,r_2\in (0,1)$ and $a_1,a_2\in (1,\infty);$
\item[(H2)] $l_i$ is Lebesgue measurable and integrable such that $\int_{\mathbb{R}}l_i(y)dy=1$ and $\int_{\mathbb{R}}l_i(y)e^{\lambda y}dy<\infty$ for all $\lambda\in \mathbb{R}, i=1,2;$
\item[(H3)] $l_i(y)=l_i(-y)\ge 0, y\in \mathbb{R},i=1,2.$
\end{description}

To apply the theory of monotone semiflows, we make a change of variable
\[
U_{n}(x)=1-U^*_n(x),V_n(x)=V^*_n(x)
\]
and drop the star for the sake of simplicity, then \eqref{1} becomes
\begin{equation}\label{4}
\begin{cases}
U_{n+1}(x)=1-\int_{\mathbb{R}}
(1-U_{n}(y))e^{r_{1}(U_{n}(y)-a_{1}V_{n}(y))}l_{1}(x-y)dy,\\
V_{n+1}(x)=\int_{\mathbb{R}}
V_{n}(y)e^{r_{2}(1-a_2-V_{n}(y)+a_{2}U_{n}(y))}l_{2}(x-y)dy.
\end{cases}
\end{equation}
Then the corresponding difference system of \eqref{4} is
\begin{equation}\label{5}
\begin{cases}
u_{n+1}=
1-(1-u_{n})e^{r_{1}(u_{n}-a_{1}v_{n})},n\in\mathbb{N},\\
v_{n+1}=
v_{n}e^{r_{2}(1-a_2-v_{n}+a_{2}u_{n})},n\in\mathbb{N},
\end{cases}
\end{equation}
and \eqref{5} has four equilibria
\[
F_0=(0,0),F_1=(1,0), F_2=(1-k_1,k_2), F_3=(1,1).
\]
Clearly, $F_0$ and $F_3$ are stable while $F_1,F_2$ are unstable. Then it suffices to study the bistable traveling wave solutions of \eqref{4} connecting $F_0$ with $F_3$. We now give the definition of traveling wave solutions as follows.
\begin{definition}{\rm
A traveling wave solution of \eqref{1} is a special solution
of the form $U_n(x)=\phi (t), V_n(x)=\psi(t), t=x+cn$ with the wave
speed $c\in \mathbb{R}$ and the wave profile $(\phi,\psi)\in \mathcal{C}$. Then
$(\phi,\psi)$ and $c$ must satisfy the following recursion system
\begin{equation}\label{6}
\begin{cases}
\phi (t+c)=1-\int_{\mathbb{R}}(1-\phi (y))e^{r_{1}(\phi (y)-a_{1}\psi(y))}l_{1}(t-y)dy,\\
\psi
(t+c)=\int_{\mathbb{R}} \psi (y)e^{r_{2}(1-a_2-\psi (y)+a_{2}\phi (y))}l_{2}(t-y)dy, t\in\mathbb{R}.
\end{cases}
\end{equation}}
\end{definition}

For a bistable traveling wave solution $(\phi,\psi)$, it satisfies
\begin{equation}\label{7}
\lim_{t\to -\infty}(\phi(t),\psi(t))=(0,0)=:\theta,\lim_{t\to \infty}(\phi(t),\psi(t))=(1,1)=:\textbf{1}.
\end{equation}

In what follows, we shall investigate the existence of \eqref{6}-\eqref{7} by the theory in Fang and Zhao \cite{fang}.
We first recall the results of \cite[Theorem 3.1]{fang}. Let ${\theta}\ll {M}\in\mathbb{R}^2$ and $Q$ be a map from $\mathcal{C}_{[\theta,{M}]}$ to $\mathcal{C}_{[\theta,{M}]}$ with $Q({\theta})={\theta},Q({M})={M}.$ Also let $F$ be the set of all spatially homogeneous steady states of $Q$ restricted on $[\theta,{M}].$ We now list the conditions of \cite[Theorem 3.1]{fang} as follows.
\begin{description}
\item[(A1)] (\emph{Transition invariance}) $T_y\circ Q[\Phi]=Q\circ T_y[\Phi]$ for any $\Phi \in \mathcal{C}_{[\theta,{M}]}$ and $y\in\mathbb{R},$ where $T_y[\Phi](x)=\Phi(x-y);$
\item[(A2)] (\emph{Continuity}) $Q: \mathcal{C}_{[\theta,{M}]} \to \mathcal{C}_{[\theta,{M}]}$ is continuous with respect to the compact open topology;
\item[(A3)] (\emph{Monotonicity}) $Q$ is order preserving in the sense that $Q[\Phi]\ge Q[\Psi]$ if $\Phi \ge \Psi$ with $\Phi, \Psi \in \mathcal{C}_{[\theta,{M}]};$
\item[(A4)] (\emph{Compactness}) $Q:\mathcal{C}_{[\theta,{M}]}\to \mathcal{C}_{[\theta,{M}]}$ is compact with respect to the compact open topology;
\item[(A5)] (\emph{Bistability}) Two fixed points $\theta$ and ${M}$ are strongly stable from above and below, respectively,
for the map $Q: [\theta,{M}] \to [\theta,{M}],$ that is, there exist a number $\delta >0$ and unit vectors ${E}_4,{E}_5$ with $\theta \ll {E}_4,{E}_5\ll {1}$ such that
\[
Q[\eta {E}_4] \ll \eta {E}_4, Q[{M}-\eta {E}_5] \gg {M}-\eta {E}_5, \forall \eta \in (0, \delta],
\]
and the set $F\backslash \{\theta,{M}\}$ is totally unordered;
\item[(A6)] (\emph{Counter-propagation}) For each ${I}\in F\backslash \{\theta,{M}\},$ $c^*_- ({I}, {M} )+c^*_+(\theta, {I})>0,$ where $c^*_- ({I}, {M} )$ and $c^*_+(\theta, {I})$ represent the leftward and rightward spreading speeds of monostable subsystem $\{Q^n\}_{n\ge 0}$ restricted on $\mathcal{C}_{[{I}, {M} ]}$ and $\mathcal{C}_{[\theta, {I} ]},$ respectively.
\end{description}

In Fang and Zhao \cite{fang}, under the assumptions (A1)-(A6), the existence of bistable traveling wave solutions of $\{Q_n\}_{n\ge 0}$ connecting $\theta$ with $M$ has been proved, which is monotone increasing.

\section{Existence of Bistable Traveling Wave Solutions}

\noindent

In this section, we consider the existence of \eqref{6}-\eqref{7}. We first present the main conclusion of this paper as follows.
\begin{theorem}\label{th1}
Assume that (H1)-(H3) hold. Then there exist $c\in\mathbb{R}$ and $(\phi,\psi)\in \mathcal{C}_{[\theta,\textbf{1}]}$ satisfying \eqref{6}-\eqref{7}, which is monotone increasing and is a desired bistable traveling wave solution of \eqref{4}.
\end{theorem}

To prove Theorem \ref{th1}, we  define
$Q=(Q_1,Q_2)$ by
\begin{equation}\label{8}
\begin{cases}
Q_1(\phi,\psi)(t)=1-\int_{\mathbb{R}}(1-\phi (y))e^{r_{1}(\phi (y)-a_{1}Y_{n}(y))}l_{1}(t-y)dy,\\
Q_2(\phi,\psi)(t)=\int_{\mathbb{R}} \psi (y)e^{r_{2}(1-a_2-\psi (y)+a_{2}\phi (y))}l_{2}(t-y)dy
\end{cases}
\end{equation}
for $\Phi=(\phi,\psi)\in \mathcal{C}_{[\theta,\textbf{1}]}.$

By Fang and Zhao \cite[Theorem 3.1]{fang}, it suffices to verify that (A1)-(A6) to prove Theorem  \ref{th1}, and we now check them by several lemmas, throughout which we always impose (H1)-(H3).

\begin{lemma}\label{le1}
If $Q$ is defined by \eqref{8}, then ${M}=(1,1)$ and
\[
F=\{F_0, F_1, F_2, F_3\}.
\]
\end{lemma}

\begin{lemma}\label{le2}
If $Q$ is defined by \eqref{8}, then it satisfies (A1).
\end{lemma}
\begin{proof}
For any $y\in \mathbb{R}$ and $(\phi,\psi)\in \mathcal{C}_{[\theta,\textbf{1}]},$ we have
\begin{eqnarray*}
T_{y}[Q_{2}(\phi ,\psi )(t)] &=&T_{y}\left[ \int_{\mathbb{R}}\psi
(s)e^{r_{2}(1-a_{2}-\psi (s)+a_{2}\phi (s))}l_{2}(t-s)ds\right]  \\
&=&T_{y}\left[ \int_{\mathbb{R}}\psi (t-s)e^{r_{2}(1-a_{2}-\psi
(t-s)+a_{2}\phi (t-s))}l_{2}(s)ds\right]  \\
&=&\int_{\mathbb{R}}\psi (t-y-s)e^{r_{2}(1-a_{2}-\psi (t-y-s)+a_{2}\phi
(t-y-s))}l_{2}(s)ds \\
&=&\int_{\mathbb{R}}T_{y}[\psi (t-s)]e^{r_{2}(1-a_{2}-T_{y}[\psi
(t-s)]+a_{2}T_{y}[\phi (t-s)])}l_{2}(s)ds \\
&=&Q_{2}(T_{y}[\phi ],T_{y}[\psi ])(t).
\end{eqnarray*}

Similarly, we obtain $T_{y}[Q_{1}(\phi ,\psi )(t)]=Q_{1}(T_{y}[\phi
],T_{y}[\psi ])(t)$. The proof is complete.
\end{proof}

\begin{lemma}\label{le3}
If $Q$ is defined by \eqref{8}, then $Q:\mathcal{C}_{[\theta,\textbf{1}]}\to \mathcal{C}_{[\theta,\textbf{1}]}$ and
satisfies (A2)-(A4).
\end{lemma}
\begin{proof}
For any $t,\delta$ and $(\phi,\psi)\in \mathcal{C}_{[\theta,\textbf{1}]},$ we have
\begin{eqnarray}
&&\left\vert Q_{1}(\phi ,\psi )(t+\delta )-Q_{1}(\phi ,\psi )(t)\right\vert
\nonumber \\
&=&\left\vert \int_{\mathbb{R}}\psi (s)e^{r_{2}(1-a_{2}-\psi (s)+a_{2}\phi
(s))}\left[ l_{2}(t+\delta -s)-l_{2}(t-s)\right] ds\right\vert   \nonumber \\
&\leq &\int_{\mathbb{R}}\psi (s)e^{r_{2}(1-a_{2}-\psi (s)+a_{2}\phi
(s))}\left\vert l_{2}(t+\delta -s)-l_{2}(t-s)\right\vert ds  \nonumber \\
&\leq &\int_{\mathbb{R}}\left\vert l_{2}(t+\delta -s)-l_{2}(t-s)\right\vert
ds,  \label{9}
\end{eqnarray}
which implies the equicontinuity of $Q(\phi,\psi)(t)$ by (H2).

Since $r_{2}\in (0,1),$ we know that $ue^{r_{2}(1-a_{2}-u+a_{2}v)}$ is
monotone increasing in both $u\in \lbrack 0,1]$ and $v\in \lbrack 0,1]$ such
that
\[
0\leq ue^{r_{2}(1-a_{2}-u+a_{2}v)}\leq 1,u\in \lbrack 0,1],v\in \lbrack 0,1],
\]%
which further implies that%
\begin{eqnarray*}
0 &=&\int_{\mathbb{R}}0\cdot l_{2}(t-s)ds \\
&\leq &\int_{\mathbb{R}}\psi (s)e^{r_{2}(1-a_{2}-\psi (s)+a_{2}\phi
(s))}l_{2}(t-s)ds \\
&\leq &\int_{\mathbb{R}}1\cdot l_{2}(t-s)ds=1
\end{eqnarray*}%
for any $(\phi,\psi)\in \mathcal{C}_{[\theta,\textbf{1}]}.$ By a similar analysis of $Q_1,$ we can prove that
$Q:\mathcal{C}_{[\theta,\textbf{1}]}\to \mathcal{C}_{[\theta,\textbf{1}]}.$

Due to the continuity and monotonicity of
\[
ue^{r_{2}(1-a_{2}-u+a_{2}v)}, 1-(1-u)e^{r_{1}(u-a_{1}v)}, u,v\in [0,1],
\]
and the verification of \eqref{9}, (A2)-(A4) are clear and we omit the proof here. The proof is complete.
\end{proof}

\begin{lemma}\label{le4}
(A5) is true.
\end{lemma}
\begin{proof}
Since $Q$ is order preserving, it suffices to consider the initial value problem when the initial value of \eqref{4} is spatially homogeneous, namely, the stability of equilibria of \eqref{5}. Then the stability result is clear by that of \eqref{2}, and we omit the details here. Moreover, $F_1$ and $F_2$ are unordered. The proof is complete.
\end{proof}

\begin{lemma}\label{le5}
$c^*_- (F_1, F_3 )+c^*_+(F_0, F_1)>0.$
\end{lemma}
\begin{proof}
To compute $c^*_- (F_1, F_3 ),$ we consider the spreading speed of
\begin{equation}\label{12}
p_{n+1}(x)=\int_{\mathbb{R}}
p_{n}(y)e^{r_{2}(1-p_{n}(y))}l_{2}(x-y)dy.
\end{equation}
By (H1)-(H3) and Hsu and Zhao \cite[Theorem 2.1]{hsuzhao}, then $c^*_- (F_1, F_3 )$ equals to the spreading speed of $q_n(x)$ defined by \eqref{12} if the initial value $p_0(x)$ is uniformly continuous and admits nonempty compact support such that
\[
0\le p_0(x)\le 1, x\in\mathbb{R},
\]
we then have
\[
c^*_- (F_1, F_3 )=\inf_{\mu >0}\frac{\ln (e^{r_2}\int_{\mathbb{R}}e^{\mu y} l_2(y)dy)}{\mu},
\]
which implies  that $c^*_- (F_1, F_3 )>0$ by (H2) and Liang and Zhao \cite[Lemma 3.8]{liang}.

To establish $c^*_+(F_0, F_1),$ define
\begin{equation}\label{10}
q_{n+1}(x)=1-\int_{\mathbb{R}}
(1-q_{n}(y))e^{r_{1}q_{n}(y)}l_{1}(x-y)dy.
\end{equation}

Let $w_n(x)=1-q_n(x),$ then \eqref{10} becomes
\begin{equation}\label{10}
w_{n+1}(x)=\int_{\mathbb{R}}
w_{n}(y)e^{r_{1}(1-w_{n}(y))}l_{1}(x-y)dy.
\end{equation}
Then $c^*_+(F_0, F_1)$ equals to the spreading speed of \eqref{10} when $w_0(x)$ admits nonempty compact support and is uniformly continuous such that
\[
0\le w_0(x)\le 1, x\in \mathbb{R},
\]
and so
\[
c^*_+(F_0, F_1)=\inf_{\mu >0}\frac{\ln (e^{r_1}\int_{\mathbb{R}}e^{\mu y} l_1(y)dy)}{\mu}>0.
\]

The proof is complete.
\end{proof}

\begin{lemma}\label{le6}
$c^*_- (F_2, F_3 )+c^*_+(F_0, F_2)>0.$
\end{lemma}
\begin{proof}
We first consider $c^*_- (F_2, F_3 ),$ let
\[
p_n(x)=U_n(x)-(1-k_1), q_n(x)=V_n(x)-k_2,
\]
then \eqref{4} leads to
\begin{equation}\label{11}
\begin{cases}
p_{n+1}(x)=k_1+\int_{\mathbb{R}}(p_n(y)-k_1)e^{r_1(p_n(y)-a_1q_n(y))}l_1(x-y)dy,\\
q_{n+1}(x)=-k_2+\int_{\mathbb{R}}(q_n(y)+k_2)e^{r_2(a_2p_n(y)-q_n(y))}l_2(x-y)dy.
\end{cases}
\end{equation}
Consider the corresponding initial value problem of \eqref{11} with
\[
0\le p_0(x)\le k_1, 0\le q_0(x)\le 1-k_2, x\in\mathbb{R},
\]
in which $p_0(x),q_0(x)$ are uniformly continuous and admit nonempty compact supports. If $0\le u\le k_1, 0\le v\le 1-k_2,$ then
\[
0\le k_1+ (u-k_1)e^{r_1(u-a_1v)}\le k_1, 0\le -k_2 +(v+k_2)e^{r_2(a_2u-v)}\le 1-k_2
\]
and both of them are monotone increasing in $u\in [0,k_1], v\in [0,1-k_2].$ Using the comparison principle of monotone semiflows, we obtain $(p_n(x),q_n(x))\in \mathcal{C},n\in \mathbb{N}$ such that
\[
0\le p_n(x)\le k_1, 0\le q_n(x)\le 1-k_2, x\in\mathbb{R},n\in\mathbb{N}.
\]

For $\mu \ge 0,$  define
\[
B_{\mu }=\left[
\begin{array}{cc}
(1-r_{1}k_{1})\int_{\mathbb{R}}e^{\mu y}l_{1}(y)dy & a_{1}r_{1}k_{1}\int_{%
\mathbb{R}}e^{\mu y}l_{2}(y)dy \\
a_{2}r_{2}k_{2}\int_{\mathbb{R}}e^{\mu y}l_{1}(y)dy & (1-r_{2}k_{2})\int_{%
\mathbb{R}}e^{\mu y}l_{2}(y)dy%
\end{array}%
\right].
\]%
We now consider the principle eigenvalue of $B_{\mu },$ denoted by $\lambda
\left( B_{\mu }\right) .$ Without loss of generality, we assume that
\[
\int_{\mathbb{R}}e^{\mu y}l_{1}(y)dy\leq \int_{\mathbb{R}}e^{\mu
y}l_{2}(y)dy,
\]%
then%
\begin{eqnarray*}
&&\left\vert
\begin{array}{cc}
\int_{\mathbb{R}}e^{\mu y}l_{1}(y)dy-(1-r_{1}k_{1})\int_{\mathbb{R}}e^{\mu
y}l_{1}(y)dy & -a_{1}r_{1}k_{1}\int_{\mathbb{R}}e^{\mu y}l_{2}(y)dy \\
-a_{2}r_{2}k_{2}\int_{\mathbb{R}}e^{\mu y}l_{1}(y)dy & \int_{\mathbb{R}%
}e^{\mu y}l_{1}(y)dy-(1-r_{2}k_{2})\int_{\mathbb{R}}e^{\mu y}l_{2}(y)dy%
\end{array}%
\right\vert  \\
&=&r_{1}k_{1}\int_{\mathbb{R}}e^{\mu y}l_{1}(y)dy\left\vert
\begin{array}{cc}
1 & -a_{1}\int_{\mathbb{R}}e^{\mu y}l_{2}(y)dy \\
-a_{2}r_{2}k_{2} & \int_{\mathbb{R}}e^{\mu y}l_{1}(y)dy-(1-r_{2}k_{2})\int_{%
\mathbb{R}}e^{\mu y}l_{2}(y)dy%
\end{array}%
\right\vert  \\
&=&r_{1}k_{1}\int_{\mathbb{R}}e^{\mu y}l_{1}(y)dy\left[ \int_{\mathbb{R}%
}e^{\mu y}l_{1}(y)dy-\int_{\mathbb{R}}e^{\mu
y}l_{2}(y)dy+(1-a_{1}a_{2})r_{2}k_{2}\int_{\mathbb{R}}e^{\mu y}l_{2}(y)dy%
\right]  \\
&\leq &(1-a_{1}a_{2})r_{2}k_{2}r_{1}k_{1}\int_{\mathbb{R}}e^{\mu
y}l_{1}(y)dy\int_{\mathbb{R}}e^{\mu y}l_{2}(y)dy \\
&<&0,
\end{eqnarray*}%
which implies that
\[
\lambda \left( B_{\mu }\right) > \int_{\mathbb{R}}e^{\mu
y}l_{1}(y)dy>1.
\]
At the same time, it is clear that
\[
\lambda \left( B_{0 }\right) >1.
\]

By what we have done, we obtain that
\[
\inf_{\mu >0}\frac{\ln (\lambda (B_{\mu}))}{\mu}>0,
\]
and so
\[
c^*_- (F_2, F_3 ) >0
\]
by Weinberger et al. \cite[Lemma 3.1]{weinberger}.

Similarly, we can prove that $c^*_+(F_0, F_2)>0.$ The proof is complete.
\end{proof}

Using Fang and Zhao \cite[Theorem 3.1]{fang}, we have proved Theorem \ref{th1}.

\end{document}